\documentclass[12pt]{amsart}
\usepackage[backend=biber,style=alphabetic,sorting=nyt,doi=false,url=false,isbn=false,giveninits=true]{biblatex}
\usepackage{hyperref}
\usepackage{tabularx,booktabs}
\usepackage{caption}
\usepackage{amsmath} 
\usepackage{bbm}
\usepackage{amsfonts}
\usepackage{amscd}
\usepackage{amsthm}
\usepackage{amssymb} 
\usepackage{cancel}
\usepackage{geometry}
\usepackage{latexsym}
\usepackage{eufrak}
\usepackage{euscript}
\usepackage{epsfig}
\usepackage{graphics}
\usepackage{array}
\usepackage{enumerate}
\usepackage{dsfont}
\usepackage{color}
\usepackage{wasysym}
\usepackage{multirow}
\usepackage{pdfsync}
\newcommand{\bel}[1]{\begin{equation}\label{#1}}

\newcommand{\be}{\begin{equation}}

\newcommand{\qe}{\end{equation}}

\newcommand{\R}{{\mathbb R}}
\newcommand{\N}{{\mathbb N}}
\newcommand{\Z}{{\mathbb Z}}
\newcommand{\C}{{\mathbb C}}

\newcommand{\Hmm}[1]{\leavevmode{\marginpar{\tiny%
$\hbox to 0mm{\hspace*{-0.5mm}$\leftarrow$\hss}%
\vcenter{\vrule depth 0.1mm height 0.1mm width \the\marginparwidth}%
\hbox to
0mm{\hss$\rightarrow$\hspace*{-0.5mm}}$\\\relax\raggedright #1}}}

\newtheorem{theorem}{Theorem}[section]

\newtheorem{lemma}[theorem]{Lemma}

\newtheorem{definition}[theorem]{Definition}
\newtheorem{remark}[theorem]{Remark}

\begin{document}

\title{The fourth-order Schr\"{o}dinger equation on lattices}

\author{Jiawei Cheng}
\address{Jiawei Cheng: School of Mathematical Sciences, Fudan University, Shanghai 200433, China}
\email{\href{mailto:chengjw21@m.fudan.edu.cn}{chengjw21@m.fudan.edu.cn}}

%\author{Bobo Hua}
%\address{Bobo Hua: School of Mathematical Sciences, LMNS, Fudan University, Shanghai 200433, China; Shanghai Center for Mathematical Sciences, Fudan University, Shanghai 200433, China}
%\email{\href{mailto:bobohua@fudan.edu.cn}{bobohua@fudan.edu.cn}}

\begin{abstract}
In this paper, we study the fourth-order Schr\"{o}dinger equation
\begin{equation*}
    i \partial_t u + {\Delta}^2 u - \gamma \Delta u = \pm |u|^{s-1}u
\end{equation*}
on the lattice $\Z^d$ with dimensions $d=1,2$ and parameter $\gamma \in \R$. In order to establish sharp dispersive estimates, we consider the fundamental solution as an oscillatory integral and analyze the Newton polyhedron of its phase function. Furthermore, we prove Strichartz estimates which yield the existence of global solutions to nonlinear equations with small data.
\end{abstract}

\maketitle

\section{Introduction}\label{sec-intro}

Fourth-order Schr\"{o}dinger equations on Euclidean spaces read as
\begin{equation}\label{equ-continuous fourth-order}
    i\partial_t u + \Delta^2 u + \varepsilon \Delta u  = F(u),
\end{equation}
where $u:\R \times \R^d \rightarrow \C$ is a complex-valued function and $\varepsilon \in \{-1,0,1\}$. With power type nonlinearity, they have been introduced by V. I. Karpman \cite{K96} and A. G. Shagalov \cite{KS00} to take into account the role of small fourth-order dispersion terms in the propagation of intense laser beams in a bulk medium with Kerr nonlinearity. Since then, the study of \eqref{equ-continuous fourth-order}, including the Strichartz estimates, global well-posedness and scattering in different dimensions, has lasted for years. See for instance B. Pausader \cite{P07,P09}, C. Miao et al. \cite{MXZ09,MXZ11}, C. Hao et al. \cite{HHW06}, B. Guo and B. Wang \cite{GW02}.

Discrete analogs of partial differential equations on graphs have been extensively studied in recent years. See \cite{B17,Gri18} for elliptic and parabolic equations and \cite{FT04,LX19,LX22} for wave equations.

The standard $d$-dimensional integer lattice, denoted by $\Z^d = \{(x_1,\cdots,x_d) \in \R^d: x_j \in \Z, j=1,\cdots,d\}$, is an important example of graphs. On $\Z^d$, the discrete Laplacian is defined as, for any function $u:\Z^d \rightarrow \C$,
\begin{equation*}
    \Delta u (x,t) := \sum_{y\in\Z^d ,\, d(x,y) = 1} \left( u(y,t) - u(x,t)\right),
\end{equation*}
with $d(x,y) = \sum_{k=1}^d |x_k-y_k|$.

In this article, we consider the discrete fourth-order Schr\"{o}dinger equation (DFS, in short) with Cauchy data on $\mathbb{Z}^d$
\begin{equation}\label{equ-origin fourthorder equation}
    \left\{
    \begin{aligned}
        & i \partial_t u(x,t) + {\Delta}^2 u(x,t) - \gamma \Delta u(x,t) = F(u(x,t)), \\
        & u(x,0) = f(x),
    \end{aligned}
    \right.
\end{equation}
where $x \in \mathbb{Z}^d$, $t\in \R$ and the parameter $\gamma \in \R$. For the linear model, i.e. $F = 0$ in \eqref{equ-origin fourthorder equation}, we restrict attention to solutions with adequate decay, i.e. Schwartz class, see Section \ref{ssec-basics on Z^d}. Therefore, we always assume the initial data $f$ is rapidly decreasing, so that the discrete Fourier transform can be applied to get the solution in the same function space. In what follows, we pursue sharp decay estimates for such solutions.

Obtaining dispersive inequalities for linear dispersive equations is a classical problem, which amounts to establishing a decay estimate for the $l^{\infty}$ norm of the solution in terms of time and the $l^1$ norm of the initial data. The key step is to obtain pointwise estimates to the corresponding fundamental solution.

We first recall some dispersive estimates on Euclidean space $\R^d$. It is well-known that the decay rate of the solution to Schr\"{o}dinger equation is of the order $|t|^{-\frac{d}{2}}$. While for wave equation, it decays like $|t|^{-\frac{d-1}{2}}$, see e.g. \cite{GV95,KT98}. M. Ben-Artzi et al. \cite{BKS00} proved sharp space-time decay properties for the biharmonic Schr\"{o}dinger operator. Roughly speaking, its fundamental solution can be bounded by $|t|^{-\frac{d}{4}}$. As for equations on $\Z^d$, results on decay estimates have also been obtained for low dimensions. See \cite{SK05} ($d \in \N$) for Schr\"{o}dinger equation; \cite{S98} ($d=2,3$), \cite{BCH23} ($d=4$) for wave equation; \cite{SK05} ($d=1$), \cite{BG17} ($d=2$), \cite{CI21} ($d=2,3,4$) for Klein-Gordon equation and \cite{CA23} ($d=2$) for fractional Schr\"{o}dinger equation. 

Among these models, only the fundamental solution of the discrete Schr\"{o}dinger equation has the separation-of-variables form. Therefore, the problem can be reduced to the case $d=1$. Unfortunately, we shall see that this special property fails to hold for the DFS. Hence we have to consider each dimension one-by-one.

We get the fundamental solution of \eqref{equ-origin fourthorder equation} by the discrete Fourier transform (see Section \ref{ssec-basics on Z^d}):
\begin{equation}\label{equ-fundamental solution}
    G(x,t) = \frac{1}{(2\pi)^d}\int_{\mathbb{T}^d}  e^{ix \cdot \xi+i(\omega^4+\gamma \omega^2)t} \,d\xi,\,\,\,\,\mbox{with}\,\, \omega(\xi) := \sqrt{\sum_{j=1}^{d} (2-2\cos \xi_j)}.
\end{equation}
Here $\mathbb{T}^d := [-\pi,\pi]^d$ is the torus and $x\cdot \xi = \sum_{j=1}^{d}x_j\xi_j$ is the usual inner product. Letting $x=vt$, we relate $G$ as an oscillatory integral,
\begin{equation}\label{equ-related integral}
    I(v,t) =\int_{\mathbb{T}^d} e^{i t (v\cdot \xi + \omega^4+\gamma \omega^2)}\, d\xi,
\end{equation}
with parameters $v \in \R^d,\gamma \in \R$. Throughout this paper, we let
\begin{equation*}
    \phi(v,\xi):= v\cdot \xi + \omega^4+\gamma \omega^2
\end{equation*}
be the phase function. For fixed $\gamma$, our aim is to look for optimal value $\alpha$ such that
\begin{equation}\label{equ-aim estimate}
    |I(v,t)| \leq C(1+|t|)^{\alpha},
\end{equation}
where $C$ is independent of $v$. 

The asymptotic behaviour of an oscillatory integral is closely connected to the critical points of its phase function. Hence, we first look for critical points of $\phi(v,\cdot)$, which suggests a partition of $v$. 
In Section \ref{sec-proof} we will see no critical points appear for $v$ away from the origin, then $I$ decays faster than $|t|^{-N}$ for any $N \in \mathbb{N}$ by integrating by parts. More importantly, we will show the constant in the inequality can be chosen uniformly in $v$, cf. Lemma \ref{lem-outside decay any faster}.

The situation becomes complicated for $v$ with small length. In order to get the uniform estimate, we view the problem as the stability of oscillatory integral under phase permutations. To make it clear, one observes that
\begin{equation*}
    \phi(v,\xi)=(v-v_0)\cdot \xi+\phi(v_0,\xi).
\end{equation*}
Therefore, if we get the decay rate of \eqref{equ-aim estimate} with $v=v_0$ and prove that this estimate holds uniformly for $I(v,t)$ with $v$ in a small neighborhood of $v_0$, then we get the global estimate by a finite covering on the velocity space. Near the nondegenerate critical points, $I$ decays as $|t|^{-\frac{d}{2}}$. However, some critical points are degenerate in this case, retarding further the decay rate. By degeneracy we mean the rank of Hess$_{\xi}\phi(v,\xi)$, i.e. the Hessian matrix of $\phi(v,\cdot)$ at $\xi$, is less than $d$. For such points, Newton polyhedra will play important roles in obtaining the desired decay. 

To the best of the author's knowledge, there are no well-known results on the dispersive estimates for the DFS in any dimension. In this article, we prove sharp decay estimates for $G$ in $d=1,2$, serving as the first step to the open question. Our main result is the following.
\begin{theorem}\label{theorem-main conclusion}
    There exists $C=C(d,\gamma)>0$ independent of $x$, such that 
    \begin{equation*}
        |G(x,t)| \leqslant C(1+|t|)^{-\beta} \log^{p}(2+|t|),\quad \forall\, (x,t)\in\Z^d\times\R.
    \end{equation*}
    \begin{itemize}
        \item [(a)] When $d=1$, then
        \begin{equation*}
            (\beta,p) =
            \begin{cases}
                (\frac{1}{4},0), & \text{if $~\gamma \in \{-8,0\}$}; \\
                (\frac{1}{3},0), & \text{if $~\gamma \in \R \backslash \{-8,0\}$}.
             \end{cases}
        \end{equation*}
        \item [(b)] When $d=2$, then
        \begin{equation*}
            (\beta,p) = 
            \begin{cases}
               (\frac{1}{2},1), & \text{if $~\gamma =-8$}; \\
               (\frac{1}{2},0), & \text{if $~\gamma \in \{-16,0\}$}; \\
               (\frac{3}{4},0), & \text{if $~\gamma \in \R \backslash \{-16,-8,0\}$}.
            \end{cases}
        \end{equation*}
    \end{itemize}
\end{theorem}

\begin{remark}
    For $d=2$, the estimate is sharp (see Theorem \ref{theorem-2 dim NP sharp}) in the sense that there exist $v_0 \in \R^2$ and a constant $A\neq 0$ such that 
    \begin{equation*}
            \lim_{t\rightarrow \pm \infty} \frac{t^{\beta}}{\log^{p}t} I(v_0,t) = A.
    \end{equation*}
    As for $d=1$, the estimate is also sharp as far as van der Corput lemma, cf. Theorem \ref{lemma-van der corput}, is concerned.
\end{remark}

We point out that essential difficulties appear for the DFS due to the absence of the separation-of-variables property. As the dimension increases, it is extremely difficult to determine the leading term of an oscillatory integral in most cases, hindering the extension of Theorem \ref{theorem-main conclusion} to all dimentions.

As we will see in Section \ref{ssec-1 dim OI}, van der Corput lemma is an important tool in unidimensional uniform estimates. For integrals with two variables, cf. Section \ref{ssec-2dim OI}, there exists a complete theory based on Newton polyhedra and adapted coordinates. To make the best use of them, a key strategy is to choose proper change of coordinates under which $\phi$ is expressed in an elegant way. The choice is always motivated by the corresponding form of Hessian matrix, mainly to make sure that the quadratic terms are as simple as possible. 

On the other hand, note that the parameter $\varepsilon$ in \eqref{equ-continuous fourth-order} is essentially given by $\{-1,0,1\}$ in the continuous setting. However, since the scaling argument fails to apply for \eqref{equ-origin fourthorder equation}, the parameter $\gamma$ could take values in $\R$, leading to a more complex calculation.  

Based on the dispersive inequalities for free equations, we prove Strichartz type estimates for the solutions to inhomogeneous equations. We will call $(q,r)$ a Strichartz pair if $q,r \geq 2$ and
\begin{equation}\label{equ-Strichartz pair}
    \left\{
        \begin{aligned}
            \frac{1}{q} < \frac{1}{2}\left(\frac{1}{2}-\frac{1}{r}\right) &~~~~ \mathrm{if}~~ \gamma=-8, \\
            \frac{1}{q} \leq \frac{1}{2}\left(\frac{1}{2}-\frac{1}{r}\right) &~~~~ \mathrm{if}~~ \gamma=0 ~\mathrm{or}-16, \\
            \frac{1}{q} \leq \frac{3}{4}\left(\frac{1}{2}-\frac{1}{r}\right) &~~~~ \mathrm{otherwise}. \\
        \end{aligned}
    \right.
\end{equation}

\begin{theorem}\label{theorem-Strichartz estimate}
    Suppose that $d=2$ and $u$ is a solution to \eqref{equ-origin fourthorder equation}. Let $(q,r)$ and $(\overline{q},\overline{r})$ are both Strichartz pairs. Then $u$ satisfies the estimate
    \begin{equation*}
        ||u||_{L^q_t l^r} \leq C_{q,r,\overline{q},\overline{r}} \left(||f||_{l^2} + ||F||_{L^{\overline{q}'}l^{\overline{r}'}}\right).
    \end{equation*} 
\end{theorem}
The proof relies on a well-known result in \cite{KT98}, see Lemma \ref{lemma-tao's result}. We also use the notion of Strichartz norm
\begin{equation*}
    ||h||_{\mathcal{S}}:= \sup_{(q,r)} ||h||_{L^{q}_{t}l^r},
\end{equation*}
where the supremum is taken over all Strichartz pairs. It is interesting to observe that the supremum in Strichartz norm can always be reached by two certain space-time norms, see Remark \ref{remark-strichartz norm}.

These estimates can be used in conjunction with a contraction mapping argument to prove global well-posedness for certain nonlinear equations with small initial data, see Theorem \ref{theorem-nonlinear equation}. 

This paper is organized as follows. In Section \ref{sec-preliminary} we give basic concepts and results, including the setting of function spaces on $\Z^d$, oscillatory integrals with one or two variables and Newton polyhedra. Then we prove Theorem \ref{theorem-main conclusion} in Section \ref{sec-proof} with a careful analysis to the integral $I(v,t)$. Finally, we utilize the decay bounds to obtain Strichartz estimates and the PDE applications in Section \ref{sec-Strichartz estimates}.
%In Section \ref{sec-new} we recall some facts of Newton polyhedra and discuss results for the DW in dimensions $d=2,3,4$ and all odd $d\geqslant 5$. In Section \ref{sec-nonli}, we prove $l^p \rightarrow l^q$ estimates, Strichartz estimates and the existence of solution for nonlinear equations. 

\section{Preliminaries}\label{sec-preliminary}

\subsection{Basics on the discrete setting and the DFS}\label{ssec-basics on Z^d}

Recall that $\mathbb{Z}^d$ denotes the standard integer lattice graph in $\mathbb{R}^d$. For $p \in [1,\infty]$, $l^p(\Z^d)$ is the $l^p$-space of functions on $\Z^d$ with respect to the counting measure, which is a Banach space with the norm 
\begin{equation*}
    ||h||_{l^p} :=\left\{
    \begin{aligned}
        &\left(\sum_{x \in \Z^d} |h(x)|^p\right)^{\frac{1}{p}},\ p\in[1,\infty);\\
        &\sup_{x\in\Z^d} |h(x)|,\ p=\infty.
    \end{aligned}\right.
\end{equation*}

We shall also use $|h|_p$ to denote the $l^p$ norm of $h$ for notational convenience. The mixed space-time Lebesgue spaces $L^q_t l^r$ are Banach spaces endowed with the norms
\begin{equation*}
    ||F||_{L^q_t l^r} := \left( \int_{\R}\left(\sum_{x\in \Z^d}|F(x,t)|^r\right)^{\frac{q}{r}}\right)^{\frac{1}{q}}
\end{equation*}
for $1\leqslant q,r < \infty$, with natural modifications for $q=\infty$ or $r=\infty$. Moreover, for proper functions $h_1,h_2$ on $\mathbb{Z}^d$ we define the convolution product as
\begin{equation*}
    h_1*h_2(x) := \sum_{y \in \mathbb{Z}^d} h_1(x-y)h_2(y),\quad\forall\, x\in\Z^d.
\end{equation*}

The $l^p$ spaces are analogous to the $L^p$ spaces of functions defined on $\R^d$. Many results in $L^p$ spaces extend to the lattice such as the H\"{o}lder inequality and Young's inequality for convolution. One major difference is that the $l^p$ spaces are nested: 
$ l^p \subset l^q$,  $\forall\,1 \leqslant p \leqslant q \leqslant \infty$.

In order to obtain the solution for linear equation, an important tool is the discrete Fourier transform. 
A function $h:\Z^d \rightarrow \C$ is said to be rapidly decreasing or in Schwartz class, if $\sup_{x \in \Z^d} |x|^k|h(x)| \leq C_k$ for $k = 0,1,2,\cdots$. For such function, we define its Fourier transform by
\begin{equation*}
    \mathcal{F}(h)(\xi)=\Hat{h}(\xi) := \sum_{x\in \mathbb{Z}^d} e^{-i\xi \cdot x}h(x),~~\xi \in \mathbb{T}^d.
\end{equation*}

The inverse transform for a smooth function $h: \mathbb{T}^d \rightarrow \C$ is defined as,
\begin{equation*}
    \mathcal{F}^{-1}(h)(x)=\Check{h}(x) := \frac{1}{(2\pi)^d} \int_{\mathbb{T}^d} e^{i \xi \cdot x} h(\xi) d\xi,~~x\in \mathbb{Z}^d.
\end{equation*}
We remark here that the Fourier transform and its inverse can be defined for more general functions. The convolution property and the Plancherel identity can be proved similarly. See e.g. \cite{W11} for more facts about discrete Fourier analysis. Applying the Fourier transform to both sides of (\ref{equ-origin fourthorder equation}), we get 
\begin{equation*}
    \left\{
        \begin{aligned}
            & \partial_t \Hat{u}(\xi,t) = i(\omega^4+\gamma \omega^2) \Hat{u}(\xi,t),\\
            & \Hat{u}(\xi,0) = \Hat{f}(\xi).
        \end{aligned}
    \right.
\end{equation*}
For any fixed $\xi$, the solution to this ordinary differential equation is 
\begin{equation*}
    \Hat{u}(\xi,t) = e^{i(\omega^4+\gamma \omega^2)t} \hat{f}(\xi).
\end{equation*}
Therefore we have
\begin{equation}\label{equ-semigroup solution}
    u(x,t)=\frac{1}{(2\pi)^d} \int_{\mathbb{T}^d} e^{i \xi \cdot x+i(\omega^4+\gamma \omega^2)t} \hat{f}(\xi) d\xi.
\end{equation}

In what follows, we will also write the solution to the linear DFS as a semigroup form, i.e. $u(x,t)=e^{itL}f(x)$, coinciding with \eqref{equ-semigroup solution}.

Moreover, the solution $u$ is given by $u=e^{itL}f=G*f$ with the Green's function $G(x,t)$ defined in (\ref{equ-fundamental solution}). We point out that $G$ is related to the oscillatory integral (\ref{equ-related integral}). For $x=vt$ and $t \neq 0$,  
\begin{equation*}
    G(x,t) = \frac{1}{(2\pi)^d} I(x/t,t).
\end{equation*}

\subsection{Uniform estimates for oscillatory integrals}\label{ssec-general OI}

The following useful notions are initiated from \cite{K83}. In the sequel, let $B_{\R^d}(\xi,r)$ (resp. $B_{\C^d}(\xi,r)$) be the usual open ball in $\R^d$ (resp. $\C^d$) with center $\xi$ and radius $r$, while $\overline{B}_{\R^d}(\xi,r)$ (resp. $\overline{B}_{\C^d}(\xi,r)$) denotes its closure.

\begin{definition}
    For $r,s>0$, we define the space $\mathcal{H}_r(s)$. A complex-valued function $P\in \mathcal{H}_r(s)$ means that $P$ is holomorphic on $B_{\C^d}(0,r)$, continuous on $\overline{B}_{\C^d}(0,r)$, and $|P(w)|<s$, $\forall \, w\in \overline{B}_{\C^d}(0,r)$.
\end{definition}
    
For any given $r,s>0$, one may verify that real polynomials belong to $\mathcal{H}_{r}(s)$ if their coefficients are sufficiently small.

\begin{definition}\label{def-uniform estimate}
    Suppose that $h:\mathbb{R}^d \rightarrow \mathbb{R}$ is real-analytic at 0 and $(\beta,p)\in\R\times\N$. We write 
    \begin{equation*}
        M(h) \curlyeqprec (\beta,p)
    \end{equation*} 
    if for any $r>0$ sufficiently small, we can find $\epsilon>0$, $C>0$ and a neighbourhood $A\subset B_{\R^d}(0,r)$ of the origin such that
    \begin{equation}\label{equ-uniform}
         \bigg| \int_{\mathbb{R}^d}e^{it(h(x)+P(x))}\psi(x) dx \bigg| \leqslant C (1+|t|)^{\beta} \log^p(2+|t|) \|\psi\|_{C^N(A)}
    \end{equation}
    holds for any $t \in \R$, $\psi \in C^{\infty}_0(A)$ and $P\in \mathcal{H}_r(\epsilon)$. Here $N \in \N$ only depends on $d$, $\|\psi\|_{C^N(A)}:=\sup\{|\partial^{\gamma}\psi(x)|:x\in A,\gamma\in\N^d,|\gamma|\leq N\}.$
\end{definition}

By $M(h,a) \curlyeqprec (\beta,p)$ we mean that $h_a(x):=h(x+a)$ has a uniform estimate at 0 with exponent pair $(\beta,p)$, i.e. $M(h_a) \curlyeqprec (\beta,p)$.

\subsection{Integrals in one variable}\label{ssec-1 dim OI}

Unidimensional uniform estimate can be treated by the following van der Corput lemma, see e.g. \cite[Section 8]{S93}.

\begin{theorem}\label{lemma-van der corput}
    Suppose that $h$ is a real-valued smooth function, and $\psi$ is complex-valued and smooth. Suppose also that $|h^{(k)}(x)| \geqslant 1$ for some $k \geqslant 2$ and all $x \in (a,b)$. Then we can conclude that
    \begin{equation*}
        \bigg| \int_{a}^{b} e^{i\lambda h(x)} \psi(x)dx \bigg| \leqslant c_k \lambda^{-1/k} \left[|\psi(b)|+ \int_a^b \psi'(x)dx\right].
    \end{equation*}
    The constant $c_k = 5 \cdot 2^{k-1}-2$ is independent of $h$, $\psi$ and $\lambda$. 
\end{theorem}

As an application, we prove Theorem \ref{theorem-main conclusion} when $d=1$.
\begin{proof}
    The integral we are handling is, c.f. \eqref{equ-related integral}, 
    \begin{equation*}
        I(v,t) = \int_{-\pi}^{\pi} e^{it[(2-2\cos u)(2+\gamma -2\cos u)+vu]}du.
    \end{equation*}
    Inspired by Theorem \ref{lemma-van der corput}, we should calculate the derivatives of $\phi$, that is,
    \begin{equation*}
        \begin{aligned}
            & \phi'(u) = 2\sin u(4+\gamma-4\cos u)+v, & \phi''(u) = -16\cos^2 u + 2(4+\gamma)\cos u + 8,\\
            & \phi^{(3)}(u) = 2\sin u (-4-\gamma +16\cos u), & \phi^{(4)}(u) = 64\cos^2 u - 2(4+\gamma)\cos u -32.
        \end{aligned}
    \end{equation*}
    Based on this, we claim that the phase function satisfies
    \begin{equation*}
        \max_{u \in [-\pi,\pi]} \{|\phi'(u)|,|\phi''(u)|,|\phi^{(3)}(u)|,|\phi^{(4)}(u)|\} \geq M >0.
    \end{equation*}
    To see this, note that for any $\gamma \in \R$, the quadratic equations (on $a:= \cos u \in [-1,1]$),
    \begin{equation*}
        -16a^2+2(4+\gamma)a+8=0 \,\,\mbox{and}\,\, 64a^2-2(4+\gamma)a-32=0,
    \end{equation*}
    do not have the same root. 
    
    This estimate is sharp for $\gamma =0$ since $|\phi'(0)|=|\phi''(0)|=|\phi^{(3)}(0)|=0$ when $v=0$. The case for $\gamma=-8$ is similar at $(u,v) = (\pi,0)$. As for the other $\gamma$, we have $\max \{|\phi'(u)|,|\phi''(u)|,|\phi^{(3)}(u)||\} \geq M' >0$. In fact, $\cos u = \pm 1$ or $\frac{4+\gamma}{16}$ is not a root of $\phi''(u)=0$.
\end{proof}

\subsection{Newton polyhedra}\label{ssec-Newton polyhedra}

We give a summary on the method of Newton polyhedra, which is pioneered in \cite{V76}. 

Without loss of generality, we can and will assume that $S:\R^d \rightarrow \R$ is an analytic function defined in a neighborhood of the origin with
\begin{equation*}
    S(0)=0\,\,\,\mbox{and}\,\,\,\nabla S(0)=0. 
\end{equation*}
$S$ can be expressed as a uniformly and absolutely convergent series, so we consider the associated Taylor expansion at the origin
\begin{equation*}
    S(x) = \sum_{m \in \mathbb{N}^d} s_{m} x^{m}.
\end{equation*}
The set
\begin{equation*}
    \mathcal{T}(S) := \{m \in \mathbb{N}^d :s_{m} \neq 0 \}
\end{equation*} 
is called the Taylor support of $S$, and we will always assume $\mathcal{T}(S) \neq \varnothing$. The Newton polyhedron of $S$, denoted by $\mathcal{N}(S)$, is the convex hull of the set
\begin{equation*}
    \bigcup_{m \in \mathcal{T}(S)} \left(m + \mathbb{R}^d_{+}\right),
\end{equation*}
where $\R^d_+=\{x\in \R^d:x_j\geq 0, \, 1 \leq j\leq d\}$. Given a compact face $\mathcal{P}$ of $\mathcal{N}(S)$, we write
\begin{equation*}
    S_{\mathcal{P}}(x) = \sum_{m \in \mathcal{P}}s_{m} x^{m}. 
\end{equation*}
Note that an edge or a vertex is also considered as a face. We say that $S$ is $\mathbb{R}$-nondegenerate if $\nabla S_{\mathcal{P}}(x)$ is nonvanishing on $(\mathbb{R}\backslash\{0\})^d$ for all compact faces $\mathcal{P}$ of $\mathcal{N}(S)$, that is,
\begin{equation*}
   \bigcup_{\mathcal{P}}\left( \bigcap_{j=1}^d \left\{x:{\partial_{j}S_{\mathcal{P}}}(x)=0\right\}\right)\subset\; \bigcup_{j=1}^d \{x:x_j=0\}.
\end{equation*}

The {Newton distance} $d_S$ of $\mathcal{N}(S)$ is defined by
\begin{equation*}
        d_S = \inf \{\rho>0:(\rho,\rho,\cdots,\rho) \in \mathcal{N}(S)\} ,
\end{equation*}
and the principal face $\pi(S)$ of $\mathcal{N}(S)$ is the face of minimal dimension containing the point 
$\boldsymbol{d_S}=(d_S,d_S,\cdots,d_S)$. We shall call the series
\begin{equation*}
    S_{pr}:=\sum_{m \in \pi(S)} s_m x^m
\end{equation*}
the principal part of $S$. If $\pi(S)$ is compact, $S_{pr}$ is a polynomial; otherwise, we consider $S_{pr}$ as a formal power series.

Following \cite{IKM05} and \cite{IM11-TAMS}, we also define the order of $S$ at a point $y$, denoted by ord$S(y)$, as the smallest integer $j$ such that $D^j S(y) \neq 0$, where $D^j S$ denotes the $j$th-order total derivative of $S$. If $S: \R^2 \rightarrow \R$, we let
\begin{equation*}
    m(S):= \max_{y \in \mathbb{S}^1} \mathrm{ord}S(y)
\end{equation*}
be the maximal order of $S$ along the unit circle in $\R^2$.

In his seminal work \cite{V76}, A. N. Varchenko proved that the leading term of the asymptotic expansion of a oscillatory integral can be read by the Newton polyhedron of its phase function under certain conditions. See also \cite[Theorem 2.3]{G18}.

\begin{theorem}\label{theorem-newton polyhedra}
    Suppose $S$ is $\mathbb{R}$-nondegenerate and $\psi: \mathbb{R}^d \rightarrow \mathbb{R}$ is smoothly supported close enough to the origin. Let $k_0\in\N$ be the greatest codimension over all faces of $\mathcal{N}(S)$ containing the point $\boldsymbol{d_S}$, i.e. $k_0=d-dim(\pi(S))$, then we have
    \begin{equation*}
       \left|\int_{\R^d}e^{itS(x)}\psi(x)\,dx\right|\leqslant C (1+|t|)^{-\frac{1}{d_S}} \log^{k_0-1}(2+|t|).
    \end{equation*}
\end{theorem}

\subsection{Integrals in two variables}\label{ssec-2dim OI}

In this part we focus on integrals on $\R^2$. From the definitions above, we may find that the Newton distance depends on the chosen local coordinate system in which $S$ is expressed. Based on this, the height of $S$ is defined by 
\begin{equation*}
    ht_S :=\sup \{d_{S,x}\},
\end{equation*}
where the supremum is taken over all local analytic coordinate systems preserving 0 and $d_{S,x}$ is the Newton distance in coordinates $\{x\}$. A given coordinate system $\{x_0\}$ is said to be adapted to $S$ if $d_{S,x_{0}} = ht_S$. 

For a certain function, it is necessary to ask whether the adapted coordinate system exists. Varchenko \cite{V76} proved the existence of such coordinate system for analytic functions (without multiple components) in two-dimensional case and violated the analogue in higher dimensions. Also in dimension two, I. A. Ikromov and D. M\"{u}ller \cite{IM11-TAMS} extended Varchenko's result to smooth setting. They also gave necessary and sufficient conditions for the adaptedness of a given system, see \cite[Corollary 2.3 and 4.3]{IM11-TAMS}.

\begin{theorem}\label{theorem-jugde adapted coordinate}
    The coordinates $\{x_0\}$ are adapted to $S$ if and only if one of the following conditions is satiefied:
    \begin{itemize}
        \item[(a)] $\pi(S)$ is a compact edge, and $m(S_{pr}) \leq d_{S,x_0}$.
        \item[(b)] $\pi(S)$ is a vertex.
        \item[(c)] $\pi(S)$ is an unbounded edge.
    \end{itemize}
\end{theorem}

From Theorem \ref{theorem-jugde adapted coordinate}, we deduce that if a coordinate system is not adapted to $S$, then $\pi(S)$ is a compact edge. In $\R^2$, it means that $\pi(S)$ lies on a line $\kappa_1 x_1 + \kappa_2 x_2 =1$. Without loss of generality, we may assume $\kappa_2 \geq \kappa_1 \geq 0$. Then there exists a smooth function $\psi(x_1)$ of $x_1$ with $\psi(0)=0$ such that an adapted coordinate system $\{y_1,y_2\}$ is given by $y_1=x_1, y_2=x_2 - \psi(x_1)$. See e.g. \cite[Theorem 5.1]{IM11-TAMS}. Therefore, we may and will assume that $S$ is expressed in an adapted coordinate system from now on. 

We also need to determine the Varchenko's exponent $\nu(S) \in \{0,1\}$, see \cite[Page 1294]{IM11-JFAA}, which is also called the multiplicity of $S$ in \cite{K84}. If there exists an adapted coordinate system such that $\pi(S)$ is a vertex, and if $ht_S \geq 2$, then we let $\nu(S) = 1$; otherwise, we let $\nu(S) = 0$. However, the condition for $\nu(S) = 1$ is a priori not so easily verified, but there exists a more accessible condition, see \cite[Lemma 1.5]{IM11-JFAA}.

\begin{theorem}\label{theorem-judge nu}
    The following conditions on $S$ are equivalent:
    \begin{itemize}
        \item [(a)] There exists an adapted local coordinate system such that $\pi(S)$ is a vertex.
        \item [(b)] For any adapted local coordinate system $\{y\}$, either $\pi(S)$ is a vertex, or a compact edge and $m(S_{pr})=d_{S,y}$.
    \end{itemize}
\end{theorem}

For oscillatory integrals in two variables, we don't have to verify the $\R$-nondegenerate condition for the phase as in Theorem \ref{theorem-newton polyhedra}. Under the assumption that $S$ is expressed in an adapted coordinate system, we can get the decay estimate below, see e.g. \cite[Theorem 1.1]{IM11-JFAA}. 

\begin{theorem}\label{theorem-2 dim newton polyhedra}
    Suppose $S:\R^2 \rightarrow \R$ and $\psi: \mathbb{R}^2 \rightarrow \mathbb{R}$ is smoothly supported close enough to the origin. Let $ht_S$ and $\nu(S)$ be defined as above, then we have
    \begin{equation*}
       \left|\int_{\R^2}e^{itS(x)}\psi(x)\,dx\right|\leqslant C (1+|t|)^{-\frac{1}{ht_S}} \log^{\nu(S)}(2+|t|).
    \end{equation*}
\end{theorem}

The next result shows that in most cases, the conclusion in Theorem \ref{theorem-2 dim newton polyhedra} is sharp.  

\begin{theorem}\label{theorem-2 dim NP sharp}
    Let $S$, $\psi$, $ht_S$ and $\nu(S)$  be as in Theorem \ref{theorem-2 dim newton polyhedra}. If $\pi(S)$, when given in adapted coordinates, is a compact edge or a vertex, then the following limits
    \begin{equation*}
        \lim_{t\rightarrow \pm \infty} \frac{|t|^{1/ht_{S}}}{\log^{\nu(S)}|t|} \int_{\R^2}e^{itS(x)}\psi(x)\,dx = c_{\pm} \psi(0)
    \end{equation*}
    exist, where the constants $c_{\pm}$ are nonzero and depend on the phase function $S$ only.
\end{theorem}

This result is proved in \cite[Theorem 1.3]{IM11-JFAA}, see also \cite[Theorem 1.2]{G09}. In particular, if $\pi(S)$ is unbounded, the estimate in Theorem \ref{theorem-2 dim newton polyhedra} may fail to be sharp. See \cite[Remark 1.4(c)]{IM11-JFAA} for a counterexample. Furthermore, the following result on stability of oscillatory integrals is valid, see \cite[Theorem 2.1]{K84}.

\begin{theorem}\label{theorem-2dim NP uniform}
    Let $S$, $\psi$, $ht_S$ and $\nu(S)$  be as in Theorem \ref{theorem-2 dim newton polyhedra}. Moreover, if the inequality
    \begin{equation*}
        \left| \int_{\mathbb{R}^2} e^{itS(x)} \psi(x)dx \right| \leqslant C(1+|t|)^{\beta_S} \log^{p_S}(2+|t|)
    \end{equation*}
    holds with some exponent pair $(\beta_S,p_S)$,
    then $M(S) \curlyeqprec (\beta_S,p_S)$.
    In particular, $M(S) \curlyeqprec (-1/{ht_S},\nu(S))$.
\end{theorem}

\section{Proof of Theorem \ref{theorem-main conclusion}}\label{sec-proof}

In this section we analyze $I(v,t)$ in detail. First, we write $I(v,t)$ as an integral on $\R^d$ by introducing a nonnegative function $\eta_0 \in C_0^{\infty}(\R^d)$ supported in $(-2\pi,2\pi)^d$. Besides, it should be non-vanishing on some neighborhood of $\mathbb{T}^d$ such that
\begin{equation*}
    \sum_{j\in\Z^d} \eta_0 (x+2\pi j)= 1,\;\forall\,x\in\R^d.
\end{equation*}
Noticing that $\omega$ is periodic and $e^{2\pi i x\cdot j}=1$, we have
\begin{equation}\label{equ-from torus to Rn}
    I(v,t)= \sum_{j\in\Z^d} \int_{\mathbb{T}^d}  e^{it\phi(v,\xi)} \, \eta_0(\xi+2\pi j) d\xi= \int_{\R^d}e^{it\phi(v,\xi)} \eta_0(\xi) d\xi.
\end{equation}

We then analyze the critical points of the phase,
\begin{equation*}
    \phi(v,\xi) = v\cdot \xi + \omega^4 + \gamma \omega^2:=v\cdot \xi + \varphi(\xi).
\end{equation*}
A direct calculation gives
\begin{equation}\label{equ-nabla of the phase}
    \nabla_{\xi} \phi(v,\xi)=\nabla \varphi(\xi)+ v = (4\omega(\xi)^2+2\gamma) \left(\sin\xi_1,\cdots,\sin\xi_d\right)+v,\ \xi\in\mathbb{T}^d.
\end{equation}
For any $v\in\R^d$, we define
\begin{equation*}
    \mathcal{C}_{v}:=\{\xi \in \mathbb{T}^d\,:\, \nabla_{\xi}\phi(v,\xi)=0\}.
\end{equation*}
By \eqref{equ-nabla of the phase}, a necessary condition for $\mathcal{C}_v \neq \varnothing$ is that 
\begin{equation}\label{equ-|v|<c}
    |v|^2 = |\nabla \varphi|^2 = (4\omega(\xi)^2+2\gamma)^2 \sum_{j=1}^d \sin^2 \xi_j \leq d(16d+2|\gamma|)^2:=\mathbf{c}^2.
\end{equation}
Note that $\mathbf{c}>0$ is a constant as long as the dimension $d$ and the parameter $\gamma$ are fixed. Therefore, $|v|>\mathbf{c} $ implies $\mathcal{C}_v= \varnothing$.

Next, we set, for $k=0,1,\cdots, d$, 
\begin{equation*}
    \Sigma_{k}=\Sigma_{k}(v,d):=\{\xi \in \R^d:  \mbox{corank\,Hess}\varphi(\xi)=k \},
\end{equation*}
and let $\Omega_k=\Omega_{k}(d)$ collect all velocities $v$ satisfying $\nabla_{\xi}\phi(v,\xi_0)=0$ for at least one $\xi_0 \in \Sigma_k$. Then, the set of degenerate critical points $\Sigma$ and the corresponding velocity set $\Omega$ can be formulated as 
\begin{equation*}
\Sigma:=\bigcup_{k=1}^{d}\Sigma_{k},\;\;\mbox{and}\;\;\;
\Omega:=\bigcup_{k=1}^{d} \Omega_{k}.
\end{equation*} 

Now we give the characterization of $\Sigma$. Only the first quadrant $[0,\pi]^d$ will be considered by symmetry.

\begin{lemma}\label{lem-characterization of degenerate}
    Suppose that $d = 2$ and $\xi=(\xi_1,\xi_2)\in [0,\pi]^2$. In what follows, we set  
    \begin{equation*}
        D_{\gamma} = \left\{\xi \in [0,\pi]^2:2\omega^2+\gamma=0\right\}
    \end{equation*}
    and
    \begin{equation*}
        E_{\gamma} = \left\{\xi \in [0,\pi]^2\backslash \left\{\left(\frac{\pi}{2},\frac{\pi}{2}\right)\right\}\,\,:\sum_{j=1}^2 \left(8\cos \xi_j - 4\sec \xi_j\right) = 8+\gamma\right\}.
    \end{equation*}
    
    \begin{itemize}

    \item[(a)] If $\gamma=0$, then 
    \begin{equation*}
        \Sigma_2 = \{(0,0)\};\,\,\Sigma_1 = \left\{\left(\frac{\pi}{2},\frac{\pi}{2}\right)\right\} \cup E_0 \backslash \{(0,0)\}.
    \end{equation*}     
    \item[(b)] If $\gamma=-8$, then 
    \begin{equation*}
        \Sigma_2 = \{(0,\pi),(\pi,0)\};\,\,\Sigma_1 = D_{-8} \cup \left\{\left(\frac{\pi}{2},\frac{\pi}{2}\right)\right\} \cup E_{-8}\backslash \{(0,\pi),(\pi,0)\}.
    \end{equation*}
    \item[(c)] If $\gamma=-16$, then
    \begin{equation*}
        \Sigma_2 = \{(\pi,\pi)\};\,\,\Sigma_1 = \left\{\left(\frac{\pi}{2},\frac{\pi}{2}\right)\right\} \cup E_{-16}\backslash \{(\pi,\pi)\}.
    \end{equation*}    
    \item[(d)] If $\gamma \in (-16,-8) \cup (-8,0)$, then
    \begin{equation*}
        \Sigma_2 = \varnothing;\,\,\Sigma_1 = D_{\gamma} \cup \left\{\left(\frac{\pi}{2},\frac{\pi}{2}\right)\right\} \cup E_{\gamma}.
    \end{equation*}
    \item[(e)] If $\gamma \in (-\infty,-16) \cup (0,\infty)$, then
    \begin{equation*}
        \Sigma_2 = \varnothing;\,\,\Sigma_1 = \left\{\left(\frac{\pi}{2},\frac{\pi}{2}\right)\right\} \cup E_{\gamma}.
    \end{equation*}
    \end{itemize}
\end{lemma}

\begin{proof}
    We first suppose that $d \geq 2$. Following \eqref{equ-nabla of the phase}, we have $\mbox{Hess}_{\xi} \phi (\xi) = \boldsymbol{D}(\xi)+\boldsymbol{O}(\xi)$. Here $\boldsymbol{D}$ is a diagonal matrix with elements $d_{jj}(\xi) = (4\omega^2+2\gamma) \cos \xi_j$ and the $(i,j)$ entry of $\boldsymbol{O}$ is $8\sin\xi_i \sin\xi_j$, $1 \leq i,j \leq d$. 

    After a tedious calculation, the determinant of $\mbox{Hess}_{\xi} \phi (\xi)$ is given by
    \begin{equation*}
        Det(\xi) = (4\omega^2+2\gamma)^{d-1} \left[(4\omega^2+2\gamma) \prod_{j=1}^{d} \cos \xi_j + 8\sum_{j=1}^{d} \sin^2\xi_j \prod_{l \neq j} \cos \xi_l \right].
    \end{equation*}
    We observe that $Det(\xi)$ will vanish if $4\omega^2+2\gamma =0$ or at least two components of $\xi$ are equal to $\frac{\pi}{2}$. Besides, we can rewrite it as 
    \begin{equation*}
        \begin{aligned}
            Det(\xi)
            & = (4\omega^2+2\gamma)^{d-1} \left(\prod_{j=1}^{d} \cos \xi_j\right) \left[4\omega^2+2\gamma+ 8\sum_{j=1}^{d} \frac{\sin^2 \xi_j}{\cos \xi_j}\right] \\
            & = 2(4\omega^2+2\gamma)^{d-1} \left(\prod_{j=1}^{d} \cos \xi_j\right) \left[4d+\gamma -\sum_{j=1}^{d} \left(8\cos \xi_j - \frac{4}{\cos \xi_j}\right)\right].
        \end{aligned}
    \end{equation*}
    Now for $d=2$, we deduce that $Det(\xi)=0$ if and only if $\xi \in \left\{\left(\frac{\pi}{2},\frac{\pi}{2}\right)\right\} \cup D_{\gamma} \cup E_{\gamma}$. Moreover, the only possible situation for $\mbox{rank}\,\mbox{Hess}_{\xi}\phi(\xi)=0$ is that 
    \begin{equation*}
        \sin \xi_1 = \sin \xi_2 = 4 \omega(\xi)^2+2\gamma =0.
    \end{equation*}
    Collecting all these facts, we complete the proof.
\end{proof}

As we discussed before, there is no critical points if $|v| > \mathbf{c}$. In this region, we have the following lemma.
\begin{lemma}\label{lem-outside decay any faster}
    Suppose that $d \geq 2$. Then for any $N \in \N$, there exists $C=C(d,N,\gamma)>0$ such that 
    \begin{equation*}
        |I(v,t)| \leqslant C (1+|t|)^{-N}
    \end{equation*}
    holds uniformly in $|v| \geq \mathbf{c}+1$.
\end{lemma}

\begin{proof}
    By \eqref{equ-|v|<c}, we observe that 
    \begin{equation}\label{equ-inproof1}
        |\nabla_{\xi}\phi| = |v - \nabla \varphi| \geq |v| - \mathbf{c}.
    \end{equation}
    Consider functions on $v \in \R^d$ and $\xi \in \R^d$ that are infinitely differentable in $\xi$, and let $\mathcal{C}^{\infty}$ denote the class of such functions $h$ so that
    \begin{equation*}
        |\partial_{\xi}^{\alpha} h(v,\xi)| \leq M_k, \quad \mbox{over all}\,\,|v| \geq \mathbf{c}+1,\,\xi \in \R^d\,\,\mbox{and}\,\,|\alpha| \leq k.
    \end{equation*}
    In particular, the cut-off function $\eta_0 \in \mathcal{C}^{\infty}$, cf. \eqref{equ-from torus to Rn}.

    Then we define the operator $\mathcal{L}$ by 
    \begin{equation}\label{equ-operator L}
        \mathcal{L} h = |v| \nabla_{\xi} \cdot \frac{h \nabla_{\xi} \phi}{|\nabla_{\xi} \phi|^2} = \nabla_{\xi} \cdot \frac{h \nabla_{\xi} \phi_1}{|\nabla_{\xi} \phi_1|^2},
    \end{equation}
    where $\phi_1 = \phi / |v|$. Due to \eqref{equ-inproof1}, we have $|\nabla_{\xi} \phi| \geq 1$ and $|\nabla_{\xi} \phi_1| \geq \frac{1}{\mathbf{c}+1}$ when $|v|\geq \mathbf{c}+1$. The most important property is that $\mathcal{L}$ takes functions in $\mathcal{C}^{\infty}$ to functions in $\mathcal{C}^{\infty}$. To see this, we suppose that $h \in \mathcal{C}^{\infty}$ (with bounds $M_0(h), M_1(h),$ etc.) and expand the gradient in \eqref{equ-operator L} to get
    \begin{equation*}
        \begin{aligned}
            |\mathcal{L} h(v,\xi)| 
            & = |v|\Bigg|  \sum_{k=1}^{d} \frac{\frac{\partial h}{\partial \xi_k} \frac{\partial \phi}{\partial \xi_k}}{|\nabla_{\xi}\phi|^2}+\frac{h\frac{\partial^2\phi}{\partial \xi_k^2}}{|\nabla_{\xi}\phi|^2} + 2h \frac{\partial \phi}{\partial \xi_k} \frac{\sum_{l=1}^{d} \frac{\partial^2 \phi}{\partial \xi_k \partial \xi_l}}{|\nabla_{\xi}\phi|^4} \Bigg| \\
            & \leq |v| \sum_{k=1}^{d} \left(\frac{M_1(h)}{|\nabla_{\xi}\phi|} + \frac{\partial^2\phi}{\partial \xi_k^2}\frac{M_0(h)}{|\nabla_{\xi}\phi|^2} + 2\sum_{l=1}^{d} \frac{M_0(h)}{|\nabla_{\xi}\phi|^3}\frac{\partial^2 \phi}{\partial \xi_k \partial \xi_l}\right) \\
            & \leq \tilde{C}_{d,\gamma,M_0(h),M_1(h)} |v| \left( \frac{1}{|\nabla_{\xi}\phi|}+ \frac{1}{|\nabla_{\xi}\phi|^2}+\frac{1}{|\nabla_{\xi}\phi|^3}\right) \\
            & = \tilde{C} \left( \frac{1}{|\nabla_{\xi} \phi_1|}+ \frac{1}{|v||\nabla_{\xi} \phi_1|^2}+\frac{1}{|v|^2|\nabla_{\xi} \phi_1|^3}\right) \\
            & \leq 3\tilde{C}(\mathbf{c}+1).
        \end{aligned}
    \end{equation*}
    In the third line we have used the facts: 
    \begin{equation*}
        \begin{aligned}
            & \frac{\partial^2\phi}{\partial \xi_k^2} = (4\omega^2+2\gamma)\cos \xi_k +8 \sin^2 \xi_k \leq 16d+2|\gamma|+8, \\
            & \sum_{l=1}^{d} \frac{\partial^2 \phi}{\partial \xi_k \partial \xi_l} = (4\omega^2+2\gamma)\cos \xi_k+8\sum_{l=1}^{d}\sin \xi_k \sin \xi_l \leq 24d+2|\gamma|.
        \end{aligned}
    \end{equation*}
    Above all, we have shown that $\mathcal{L}h$ can be bounded by a constant which is independent of $v$. The rest estimates about the derivatives of $\mathcal{L}h$ are similar. 
    
    Finally, we integrate $I$ by parts $N$ times to obtain
    \begin{equation*}
        |I| = \Bigg| \left(\frac{-1}{it|v|}\right)^N \int_{\R^2}e^{it\phi} \mathcal{L}^N (\eta_0)\,d\xi\Bigg| \leq \frac{M_N(\eta_0)}{(t|v|)^N} \leq M_N (\mathbf{c}+1)^{-N} |t|^{-N}.
    \end{equation*}
    The estimate $|I|\leq C$ holds trivially, and the proof is complete.
\end{proof}
 
Now it suffices to estimate $I$ when $v \in \mathcal{V}:=\{v \in \R^2: |v| \leq \mathbf{c}+1\}$. Due to the appearence of critical points for such $v$, a different strategy will be used, which is explained in the following lemma. Roughly speaking, in order to obtain a global uniform estimate for $I$, we only need to bound it locally in $\mathcal{V} \times \mathcal{U}$, where we denote by $\mathcal{U}$ the compact support of $\eta_0$.

\begin{lemma}\label{lemma-inside the cone}
    For any $(v_0,\xi) \in \mathcal{V} \times \mathcal{U}$, suppose that 
    \begin{equation}\label{equ-Mxi}
        M(\phi(v_0,\cdot),\xi) \curlyeqprec (\beta_{v_0,\xi},p_{v_0,\xi}) \,\,\,\,\mbox{for some}\,(\beta_{v_0,\xi},p_{v_0,\xi}) \in \R \times \N.
    \end{equation}
    Then there exist a proper exponent pair $(\beta_0,p_0)$ and $C_0=C_0(\gamma)$ such that
    \begin{equation*}
        |I(v,t)| \leq C_0(1+|t|)^{\beta_0}\log^{p_0}(2+|t|)
    \end{equation*}
    holds uniformly in $v \in \mathcal{V}$.

    In particular, the choice of $(\beta_0,p_0)$ depends closely on $(\beta_{v_0,\xi},p_{v_0,\xi})$ in \eqref{equ-Mxi}, which will be explained in the coming proof.
\end{lemma}

\begin{proof}
    We first fix $v_0 \in \mathcal{V}$. By \eqref{equ-Mxi} and Definition \ref{def-uniform estimate}, there exist $\epsilon_{v_0,\xi}>0$ and neighborhood $\mathbf{m}_{v_0,\xi}$ of $\xi$ such that \eqref{equ-uniform} holds. Therefore, the compact set $\mathcal{U}\subset\cup_{\xi\in\mathcal{U}}\mathbf{m}_{v_0,\xi}$ and we can choose finite sets $\{\mathbf{m}_{v_0,\xi_j}, j=1,2,\cdots,N_0\}$ to cover it. By a partition of unity, there are nonnegative functions $\{\varphi_j, \, j=1,2,\cdots,N_0\}$ such that
    \begin{equation*}
        \varphi_j \in C_0^{\infty}(\mathbf{m}_{v_0,\xi_j})\quad \text{and}\quad \sum_{j=1}^{N_0} \varphi_j  \equiv 1\;\;\mbox{on}\;\; \mathcal{U}.
    \end{equation*}
    We insert these functions into $I$:
    \begin{equation*}
       I(v,t) = \int_{\R^2} e^{it\phi}\;\eta_0 = \sum_{j=1}^{N_0} \int_{\R^2} e^{it\phi}\;\eta_0 \varphi_j:=\sum_{j=1}^{N_0}I^{j}(v,t).
    \end{equation*}
    Again by \eqref{equ-Mxi}, we have $M(\phi(v_0,\cdot),\xi_j) \curlyeqprec (\beta_{v_0,\xi_j},p_{v_0,\xi_j})$. Thus
    \begin{equation*}
    |I^j(v,t)| \leqslant C_j(1+|t|)^{\beta_{v_0,\xi_j}}\log^{p_{v_0,\xi_j}}(2+|t|)
    \end{equation*}
    holds uniformly for all $v$ such that $|v-v_0|<\epsilon_{v_0,j}$. Consequently, adding these inequalities together gives local estimate for $I$ near $v_0$, with index $(\beta_{v_0},p_{v_0})=\max_j\{(\beta_{v_0,\xi_j},p_{v_0,\xi_j})\}$ in the lexicographic order and $C_{v_0}=\max_{j}C_j$.

    Now, for each $v_0 \in \mathcal{V}$ we have obtained a uniform estimate of $I(v,t)$ in a small neighborhood of $v_0$, i.e. $|v-v_0| \leq \epsilon_{v_0}:= \min_{1\leq j\leq N_0}\epsilon_{v_0,j}$, then we finish the proof by a finite covering on this compact region.   
\end{proof}

In conclusion, the left part is to prove \eqref{equ-Mxi} for all $(v_0,\xi) \in \mathcal{V} \times \mathcal{U}$. Then Theorem \ref{theorem-main conclusion} is proved by Lemma \ref{lem-outside decay any faster} and \ref{lemma-inside the cone} when $d=2$.

To start with, we naturally fix some $v_0$, and classify all the $\xi$ into three parts: 
regular points (i.e. $\nabla_{\xi}\phi(v_0,\xi)\neq 0$),  nondegenerate critical points and degenerate critical points of $\phi(v_0,\cdot)$. For the first two cases, we have the following estimate.
\begin{lemma}\label{lemma-regular and nondegenerate estimate}
    Suppose that $h:\R^d \rightarrow \R$ is real analytic at some point $x_0$.
    \begin{itemize}
        \item [(a)] If $\nabla h(x_0) \neq 0$, then $M(h,x_0) \curlyeqprec (-N,0)$ for any $N \in \N$.
        \item [(b)] If $x_0$ is a nondegenerate critical point of $h$, then $M(h,x_0) \curlyeqprec (-d/2,0)$.
    \end{itemize}
\end{lemma}

\begin{proof}
    The first assertion can be proved directly through integrating by parts, while the second one can be handled in the spirit of the stationary phase method.
\end{proof}

Finally, we need to establish \eqref{equ-Mxi} for all $\xi_0\in\Sigma=\cup_{j=1}^{d}\Sigma_j$. With Lemma \ref{lem-characterization of degenerate}, we find different cases will appear as the parameter $\gamma$ varies. Thus, we will analyze each possibility one-by-one. Before the proof we give some notions:

\begin{itemize}
    \item [(N1)] The constants $\tilde{c}$ may vary from one line to another or disappear directly. This is reasonable because they contribute nothing to the decay.
    \item [(N2)] In the Taylor expansion for certain phase function $S(x)$, we use $N(x)$ to denote the negligible terms in the sense that $S$ and $S-N$ share the same Newton polyhedron, i.e. $\mathcal{N}(S) = \mathcal{N}(S-N)$.
    \item [(N3)] We will use $\xi_0 = (\xi_1^*,\xi_2^*)$ and the abbreviations below unless specified otherwise. For $j = 1,2$, $c_j: = \cos(\xi_j),\,s_j := \sin(\xi_j),\,t_j:=\tan(\xi_j), c_j^*: = \cos(\xi_j^*),\,s_j^* := \sin(\xi_j^*),\,t_j^*:=\tan(\xi_j^*)$.
\end{itemize}

%\begin{lemma} 
%    Suppose that $d=2$ and $v_0 \in \Omega$. Then there exist a neighbourhood $V$ of $v_0$ and $C>0$ such that 
%    \begin{equation*}
%        |I(v,t)|\leq C\log^{p} (2+|t|)(1+|t|)^{-\beta},\;\mbox{with}\;\;
%       (\beta,p)= 
%        \begin{cases}
%           (\frac{1}{2},0), & \text{if $~(\gamma,v_0) = (0,0)$ or $(-8,0)$}; \\
%            (\frac{3}{4},0), & \text{if $~\gamma \in [-16,0]$ and $v_0 \in \mathcal{D}_{\gamma}\backslash \{0\}$}; \\
%            (\frac{3}{4},0), & \text{if $~\gamma \in \R \backslash \{-8\}$ and $v_0 = (16+2\gamma)(-1,-1)$};\\
%            (\frac{5}{6},0), & \text{if $~\gamma \in \R $ and $v_0 \in \mathcal{E}_{\gamma}$},
%        \end{cases}
%    \end{equation*}
%    holds uniformly in $v \in V $.
%\end{lemma}

\subsection{$\gamma \in (-\infty,-16) \cup (0,\infty)$}

We split the proof into two parts.

\begin{theorem}\label{theorem-xi = pi/2}
    Suppose that $\gamma \in \R\backslash \{-8\}$, $\xi_0 = \left(\frac{\pi}{2},\frac{\pi}{2}\right)$ and $v_0=-\nabla \varphi(\xi_0)=(16+2\gamma)(-1,-1)$. Then
    \begin{equation*}
        M(\phi(v_0,\cdot),\xi_0) \curlyeqprec (-3/4,0).
    \end{equation*}
\end{theorem}
\begin{proof}
    The phase function $\phi(\xi) = \omega^4+\gamma \omega^2-(16+2\gamma)\xi_1-(16+2\gamma)\xi_2$. We apply a linear transform $2u_1=\xi_1+\xi_2-\pi, 2u_2=\xi_1-\xi_2$ to obtain $\phi(u_1,u_2) = (4+4\sin u_1 \cos u_2)(4+\gamma+4\sin u_1 \cos u_2)-4(8+\gamma)u_1+\tilde{c}$. The Taylor expansion of $\phi$ at $(0,0)$ is given by
    \begin{equation*}
        \begin{aligned}
            \phi(u) & = 16\left[\frac{8+\gamma}{4} u_1 + u_1^2 -\frac{1}{2!}\frac{8+\gamma}{4}u_1u_2^2-\frac{1}{3!}\frac{8+\gamma}{4}u_1^3 + o(|u|^3)\right]-4(8+\gamma)u_1+\tilde{c} \\
            & = 16u_1^2-2(8+\gamma)u_1u_2^2+N(u).
        \end{aligned}
    \end{equation*}
    We point out that all the terms without $u_1$, i.e. $u_2^k$, do not appear in the expansion. Therefore, the Newton polyhedron is
    \begin{equation*}
        \mathcal{N}(\phi) = \{(\lambda+1,2-2\lambda):\lambda \in [0,1]\}+\R^2_+.
    \end{equation*}
    The principal face $\pi(\phi)$ is a compact edge. By Theorem \ref{theorem-jugde adapted coordinate}, the coordinate system $\{u_1,u_2\}$ is adapted. In fact, one may verify that $\phi_{pr} = 16u_1^2-2(8+\gamma)u_1u_2^2$ and $m(\phi_{pr}) = 1 < \frac{4}{3} = d_{\phi,u}$. Consequently, a combination of Theorem \ref{theorem-2 dim newton polyhedra} and \ref{theorem-2dim NP uniform} implies the conclusion.
\end{proof}

On the other hand, the rest situation is $\xi_0 \in E_{\gamma}$.

\begin{theorem}\label{theorem-xi in E gamma}
    Suppose that $\gamma \in \R\backslash \{-8\}$, $\xi_0 \in E_{\gamma}$ and $v_0 = -\nabla \varphi(\xi_0)$. Then 
    \begin{equation*}
        M(\phi(v_0,\cdot),\xi_0) \curlyeqprec (-5/6,0).
    \end{equation*}
\end{theorem}

\begin{proof}
    In this case we have $\xi_0=(\xi_1^*,\xi_2^*)\neq \left(\frac{\pi}{2},\frac{\pi}{2}\right)$, i.e. $c_1^*\neq0, c_2^*\neq0$, such that 
    \begin{equation}\label{equ-inproof2}
        8(c_1^*+c_2^*) - 4\left(\frac{1}{c_1^*}+ \frac{1}{c_2^*}\right) =8+\gamma= \frac{4(c_1^*+c_2^*)}{c_1^*c_2^*}(2c_1^*c_2^*-1).
        %8+\gamma = 4(c_1^*+c_2^*)-4(s_1^*t_1^*+s_2^*t_2^*).
    \end{equation}

    In this Theorem we assume additionally that ${s_1^*}^2+{s_2^*}^2 \neq 0$, see Remark \ref{remark-complement for E gamma}.

    We then look for a proper change of coordinates. The Hessian matrix of $\phi(v_0,\cdot)$ at $\xi_0$ is
    \begin{equation*}
        \frac{64}{c_1^*c_2^*}
        \begin{pmatrix}
            -c_1^*{s_2^*}^2 & s_1^*s_2^*c_2^* \\
            s_1^*s_2^*c_1^* & -{s_1^*}^2c_2^*
        \end{pmatrix}.
    \end{equation*}
    It has rank one, and hence has zero eigenvalue of multiplicity one whose eigenvector can be chosen as $(t_1^*,t_2^*)^T$. Inspired by this, we use the change of variables
    \begin{equation*}
        \begin{pmatrix}
            \xi_1 \\
            \xi_2
        \end{pmatrix}
        =
        \begin{pmatrix}
            t_1^* & -t_2^* \\
            t_2^* & t_1^*
        \end{pmatrix}
        \begin{pmatrix}
            y_1 \\
            y_2
        \end{pmatrix}
        +
        \begin{pmatrix}
            \xi_1^* \\
            \xi_2^*
        \end{pmatrix}.
    \end{equation*}
    After some calculation, we write 
    \begin{equation*}
        \phi(v_0,y) = \tilde{c} + a_1y_1^3 + a_2y_2^2 + N(y).
    \end{equation*}
    In the expansion we have
    \begin{equation*}
        \begin{aligned}
            6 a_1 & = -2(8+\gamma-4c_1^*-4c_2^*)(s_1^*{t_1^*}^3+s_2^*{t_2^*}^3)+24(s_1^*t_1^*+s_2^*t_2^*)^2 \\
            & = 8(s_1^*t_1^*+s_2^*t_2^*)(s_1^*{t_1^*}^3+s_2^*{t_2^*}^3+3s_1^*t_1^*+3s_2^*t_2^*) \\
            & = \frac{8}{c_1^*c_2^*}(c_1^*+c_2^*)(1-c_1^*c_2^*)\left(\frac{1-{c_1^*}^2}{c_1^*}({t_1^*}^2+3)+\frac{1-{c_2^*}^2}{c_2^*}({t_2^*}^2+3)\right)
            %= 8(s_1^*t_1^*+s_2^*t_2^*)\left(\frac{1-{c_1^*}^2}{c_1^*}({t_1^*}^2+3)+\frac{1-{c_2^*}^2}{c_2^*}({t_2^*}^2+3)\right)
        \end{aligned}
    \end{equation*}
    and 
    \begin{equation*}
        \begin{aligned}
            2a_2  &= 2(8+\gamma-4c_1^*-4c_2^*)(c_1^*{t_2^*}^2+c_2{t_1^*}^2)+8(-s_1^*t_2^*+s_2^*t_1^*)^2 \\
            &= -8(s_1^*t_1^*+s_2^*t_2^*)(c_1^*{t_2^*}^2+c_2^*{t_1^*}^2)+8(-s_1^*t_2^*+s_2^*t_1^*)^2 \\
            & = -8c_1^*c_2^*({t_1^*}^2+{t_2^*}^2)^2.
        \end{aligned}
    \end{equation*}
    We chaim that both $a_1$ and $a_2$ are nonvanishing given \eqref{equ-inproof2} and ${s_1^*}^2+{s_2^*}^2 \neq 0$. It is direct to see $a_2 \neq 0$ and the restriction $\gamma \neq -8$ ensures that $c_1^*+c_2^* \neq 0$. To verify that the system of equations \eqref{equ-inproof2} and 
    \begin{equation}\label{equ-inproof3}
        \frac{1-{c_1^*}^2}{c_1^*}({t_1^*}^2+3)+\frac{1-{c_2^*}^2}{c_2^*}({t_2^*}^2+3)=0=-2(c_1^*+c_2^*) + \frac{1}{c_1^*}+ \frac{1}{c_2^*} + \left(\frac{1}{{c_1^*}}\right)^3+\left(\frac{1}{{c_2^*}}\right)^3
    \end{equation} 
    has no solutions, we first observe that $c_1^*$ and $c_2^*$ have opposite signs by \eqref{equ-inproof3}. Then we insert \eqref{equ-inproof2} into \eqref{equ-inproof3} to get
    \begin{equation*}
        \left(\frac{1}{{c_1^*}}\right)^3+\left(\frac{1}{{c_2^*}}\right)^3 = \frac{8+\gamma}{4}.
    \end{equation*}
    If $\gamma>-8$, since the signs of $c_1^*$ and $c_2^*$ are different, we may assume $-c_2^* > c_1^* >  0$ without loss of generality. Thus, we have
    \begin{equation*}
        0< 8+\gamma = 8(c_1^*+c_2^*) - 4\left(\frac{1}{c_1^*}+ \frac{1}{c_2^*}\right)  < 0.
    \end{equation*}
    This contradiction implies the chaim above is valid. The situation for $\gamma < -8$ is similar.
    
    Therefore, the Newton polyhedron is 
    \begin{equation*}
        \mathcal{N}(\phi) = \{(3\lambda,2-2\lambda):\lambda \in [0,1]\} + \R^2_+.
    \end{equation*}
    In this case $\phi_{pr} = a_1y_1^3+a_2y_2^2$, $m(\phi_{pr}) \leq 1 < \frac{6}{5} = d_{\phi,y}$ and hence $\{y_1,y_2\}$ is adapted. Then Theorem \ref{theorem-2 dim newton polyhedra} and \ref{theorem-2dim NP uniform} can be applied.
\end{proof}

\begin{remark}\label{remark-complement for E gamma}
    If $\gamma \in \R \backslash \{-16,-8,0\}$, ${s_1^*}^2+{s_2^*}^2 \neq 0$ is valid for any $\xi_0 \in E_{\gamma}$. If $\gamma \in \{-16,-8,0\}$, there are only four exceptions which will be discussed later. See Section \ref{ssec-0and-16} for $\xi_0=(0,0) (\gamma=0)$ and $\xi_0=(\pi,\pi) (\gamma=-16)$. See Section \ref{ssec--8} for $\xi_0 \in \{(0,\pi),(\pi,0)\} (\gamma=-8)$.
\end{remark}

Then we apply Theorem \ref{theorem-xi = pi/2}, \ref{theorem-xi in E gamma} and Lemma \ref{lemma-regular and nondegenerate estimate} to finish the proof.

\subsection{$\gamma \in (-16,-8) \cup (-8,0)$}

Note that Theorem \ref{theorem-xi = pi/2} and \ref{theorem-xi in E gamma} can still be utilized if $\xi_0$ belongs to the corresponding set. We need to investigate the case when $\xi_0 \in D_{\gamma}$.

\begin{theorem}\label{theorem-xi in D gamma}
    Suppose that $\gamma \in (-16,-8) \cup (-8,0)$, $\xi_0 \in D_{\gamma}$ and $v_0 = -\nabla \varphi(\xi_0)$. Then 
    \begin{equation*}
        M(\phi(v_0,\cdot),\xi_0) \curlyeqprec (-5/6,0).
    \end{equation*}
\end{theorem}
\begin{proof}
    In this case, the Hessian matrix of $\phi(v_0.\cdot)$ is
    \begin{equation*}
        64s_1^*s_2^*
        \begin{pmatrix}
            s_1^* & s_2^* \\
            s_1^* & s_2^*
        \end{pmatrix}.
    \end{equation*}
    Hence the eigenvector is chosen as $(s_2^*,-s_1^*)^T$ and the change of variables is $\xi_1 = s_2^* z_1 + s_1^*z_2 + \xi_1^*,\,\xi_2 = -s_1^*z_1+s_2^*z_2 + \xi_2^*$. The expansion becomes
    \begin{equation*}
        \phi(v_0,z) = (b_1z_1^2 + 2b_2z_2)^2 + N(z).
    \end{equation*}
    The constants $b_1 = {s_2^*}^2c_1^* + {s_1^*}^2c_2^*\,,\,b_2 = {s_1^*}^2+{s_2^*}^2$ are nonvanishing. To see this, we deduce from 
    \begin{equation*}
        0=2\omega(\xi_0)^2+\gamma = 8+\gamma -4(c_1^*+c_2^*)
    \end{equation*}
    that $c_1^*c_2^* \neq 1$ and $c_1^*+c_2^* \neq 0$, which implies $b_1 = (1-c_1^*c_2^*)(c_1^*+c_2^*)\neq 0$ and $b_2 \neq 0$.

    Unfortunately, the coordinate system $\{z_1,z_2\}$ is not adapted, since $\pi(\phi)$ is a compact edge but $m(\phi_{pr}) \geq 2 > \frac{4}{3} = d_{\phi,z}$. To overcome this difficulty, we apply $w_1 = b_1z_1^2+2b_2z_2,\,w_2 = z_1$ to obtain
    \begin{equation*}
        \phi(v_0,w) = w_1^2 - \frac{4}{3}b_1^2w_2^3+N(w)
    \end{equation*}
    The left analysis is the same as the proof of Theorem \ref{theorem-xi in E gamma}.
\end{proof}

Theorem \ref{theorem-xi = pi/2}, \ref{theorem-xi in E gamma}, \ref{theorem-xi in D gamma} and Lemma \ref{lemma-regular and nondegenerate estimate} imply a decay of $\mathcal{O}(|t|^{-3/4})$.

\subsection{$\gamma \in \{0,-16\}$} \label{ssec-0and-16}
When $\gamma=0$, we will encounter the Hessian matrix with vanishing rank for the first time. We focus on $\xi_0 \in D_0 = \{(0,0)\}$, while we can still use Theorem \ref{theorem-xi in E gamma} for the other $\xi_0 \in E_0 \backslash \{(0,0)\}$.

\begin{theorem}\label{theorem-xi is 0}
    Suppose that $\gamma=0$ and $\xi_0 = (0,0)$, then $M(\phi(0,\cdot),\xi_0) \curlyeqprec (-1/2,0)$.
\end{theorem}
\begin{proof}
    We expand $\phi = \omega^4$ directly at the origin to obtain 
    \begin{equation*}
        \phi(0,\xi) = \tilde{c}+(4-2\cos \xi_1-2\cos \xi_2)^2 = (\xi_1^2+\xi_2^2)^2 + N(\xi).
    \end{equation*}
    Since $\phi_{pr} = (\xi_1^2+\xi_2^2)^2$ and $m(\phi_{pr}) = 0 < 2 = d_{\phi,\xi}$, $\{\xi_1,\xi_2\}$ itself is adapted. Furthermore, $\phi$ is also $\R$-nondegenerate under this coordinate system. The conclusion follows easily from Theorem \ref{theorem-newton polyhedra} and \ref{theorem-2dim NP uniform}.
\end{proof}
\begin{remark}
    \noindent
    \begin{itemize}
        \item [(a)] It is a little harder to determine $\nu(\phi)$ in this case. In fact, $\phi_{pr} < d_{\phi,\xi}$ violates the second condition in Theorem \ref{theorem-judge nu}, implying $\nu(\phi) = 0$. For this special phase, we can also take polar coordinates to calculate the decay directly.
        \item [(b)] The case for $\gamma = -16$ is similar. When we expand the phase $\phi(0,\cdot)$ at $(\pi,\pi)$, we get the same principal part $\phi_{pr}(\xi) = (\xi_1^2+\xi_2^2)^2$.
    \end{itemize}
\end{remark}

Then, we combine Theorem \ref{theorem-xi = pi/2}, \ref{theorem-xi in E gamma}, \ref{theorem-xi is 0} and Lemma \ref{lemma-regular and nondegenerate estimate} to finish the proof.

\subsection{$\gamma = -8$}\label{ssec--8}

The trickiest part lies in this section, so we start the analysis from the very beginning. Recall that the determinant of the Hessian matrix is
\begin{equation*}
    \begin{aligned}
        \mathrm{det} (\mathrm{Hess} \phi)
        & = (4\omega^2-16)[(4\omega^2-16)c_1c_2+8s_1^2c_2+8s_2^2c_1] \\
        & = -64(c_1+c_2)^2(1-2c_1c_2).
    \end{aligned}
\end{equation*}
Inspired by this, we split the proof into three cases.

$\bullet$ $c_1^*+c_2^*=0$ and $|c_1^*|=1$

In this case, we assume $\xi_0 = (0,\pi)$ without loss of generality and expand the phase at this point to get
\begin{equation*}
    \begin{aligned}
        \phi(0,\xi) & = (4-2\cos \xi_1 -2\cos(\xi_2+\pi))^2 -8(4-2\cos \xi_1- 2\cos(\xi_2+\pi)) \\
        & = (-4-2\cos \xi_1 +2\cos \xi_2)(4-2\cos\xi_1 +2\cos \xi_2) \\
        & = \tilde{c} + (\xi_1^2-\xi_2^2)^2 + N(\xi).
    \end{aligned}
\end{equation*}
Similarly, this coordinate system is adapted (but not $\R$-nondegenerate). To investigate the Varchenko's exponent, we let $y_1 = \xi_1+\xi_2,\,y_2 = \xi_1-\xi_2$, then 
\begin{equation*}
    \phi(y) = \tilde{c} + y_1^2y_2^2 + N(y).
\end{equation*}
$\{y_1,y_2\}$ is also adapted and the principal face is a vertex $(2,2)$. Therefore, $\nu(\phi) = 1$ by Theorem \ref{theorem-judge nu}. In conclusion, $M(\phi,\xi_0) \curlyeqprec (-1/2,1)$.

$\bullet$ $c_1^* + c_2^* = 0$ and $|c_1^*| \neq 1$

This condition implies $s_1^* = s_2^*$ and $s_1^*s_2^* \neq 0$. We apply the same change of variables as Theorem \ref{theorem-xi in E gamma} to get
\begin{equation*}
    \phi(y) = \tilde{c} + a_2 y_2^2(1 + g_1(y_1))+y_2^3 g_2(y_1,y_2) = \tilde{c} + a_2 y_2^2 + N(y)
\end{equation*} 
where
\begin{equation*}
    2a_2 = -8({c_1^*}+c_2^*)^2+8(s_1^*+s_2^*)^2 = 8(s_1^*+s_2^*)^2 \neq 0.
\end{equation*}
Then the Newton polyhedron is given by
\begin{equation*}
    \mathcal{N}(\phi) = \{(0,2)+\R^2_+\}.
\end{equation*}
By Theorem \ref{theorem-jugde adapted coordinate}, $\{y_1,y_2\}$ is adapted, and $\nu(\phi)= 0$ by Theorem \ref{theorem-judge nu}. Hence $M(\phi,\xi_0) \curlyeqprec (-1/2,0)$.

$\bullet$ $c_1^*c_2^*=\frac{1}{2}$

This condition implies $c_1^*+c_2^* \neq 0$, $c_1^* \neq 0$, $c_2^* \neq 0$ and ${s_1^*}^2+{s_2^*}^2 \neq 0$. Thus we repeat the same proof as Theorem \ref{theorem-xi in E gamma} and get the same expansion. Now $a_2 \neq 0$ is trivial. To see that $a_1 \neq 0$, we calculate
\begin{equation*}
    \begin{aligned}
        \frac{3}{4}a_1
        & = \frac{(1-c_1^*c_2^*)(c_1^*+c_2^*)}{c_1^*c_2^*} \left[\left(\frac{1}{{c_1^*}^3}+\frac{1}{{c_2^*}^3}\right)+\frac{(c_1^*+c_2^*)(1-2c_1^*c_2^*)}{c_1^*c_2^*}\right] \\
        & = (c_1^*+c_2^*)\left(\frac{1}{{c_1^*}^3}+\frac{1}{{c_2^*}^3}\right).
    \end{aligned}
\end{equation*}
Since $c_1^* \neq -c_2^*$, we deduce $a_1 \neq 0$. Hence $M(\phi,\xi_0) \curlyeqprec (-5/6,0)$.

Above all, a decay of $\mathcal{O}(|t|^{-1/2}\log|t|)$ is proved.

\section{Strichartz estimates and nonlinear equations}\label{sec-Strichartz estimates}

Strichartz estimates for fourth-order Schr\"{o}dinger equation \eqref{equ-continuous fourth-order} on $\R^d$ was established by Pausader, see \cite[Proposition 3.1]{P07}. In this section, we will prove Theorem \ref{theorem-Strichartz estimate} by a classical result in the abstract setting. In what follows, we write $h_1\lesssim h_2$ if there exists $C>0$ independent of $t\in\R$ such that $h_1\leq C h_2$.

\begin{lemma}[Theorem 1.2 in \cite{KT98}]\label{lemma-tao's result}
    Let $(X,dx)$ be a measure space and $H$ be a Hilbert space. Suppose that for each time $t \in \R$ we have an operator $U(t):H\rightarrow L^2(X)$ which obeys the energy estimate
    \begin{equation*}
        \|U(t)f\|_{L^2_x} \lesssim \|f\|_H
    \end{equation*}
    and the truncated decay estimate for some $\sigma>0$
    \begin{equation*}
        \|U(t)(U(s))^*g\|_{L^{\infty}_x} \lesssim (1+|t-s|)^{-\sigma}\|g\|_{L^1_x},
    \end{equation*}
    where $(U(t))^*$ denotes the adjoint of $U(t)$. Then the estimates
    \begin{equation*}
        \|U(t)f\|_{L^{q}_t L^{r}_x} \lesssim \|f\|_H,
    \end{equation*}
    \begin{equation*}
        \left\|\int_{-\infty}^{t} U(t)(U(s))^*F(s)ds\right\|_{L^{q}_t L^{r}_x} \lesssim \|F\|_{L^{\overline{q}'}_t L^{\overline{r}'}_x}
    \end{equation*}
    hold for all $\sigma$-admissible exponent pairs $(q,r)$, $(\overline{q},\overline{r})$, i.e. $q,r \geqslant 2$, $(q,r,\sigma) \neq (2,\infty,1)$ and $\frac{1}{q} \leqslant \sigma (\frac{1}{2}-\frac{1}{r})$.
\end{lemma}

\begin{proof}[Proof of Theorem \ref{theorem-Strichartz estimate}]
    We will prove the conclusion for $\gamma=0$, the other cases are similar. Our proof relies on the representation of the solution, which is given by Duhamel's formula,
    \begin{equation}\label{equ-Duhamel formula}
        u(x,t) = e^{itL} f(x) - i\int_0^t e^{i(t-s)L} F(x,s)ds. 
    \end{equation}
    Recall that $e^{itL}$ has been defined in \eqref{equ-semigroup solution}. Let $H=l^2(\Z^2)$, $X=\Z^2$ and the truncated operators $U_{\pm}(t)=\chi _{[0,\infty]}(t) e^{\pm itL}$ for any fixed $t \in \R$ in Lemma \ref{lemma-tao's result}, where $\chi _{[0,\infty]}(t)=1$ if $t\geq 0$ and $\chi _{[0,\infty]}(t)=0$ if $t<0$. One easily checks that $U_{\pm}(t)$ satisfy the following properties:
    \begin{itemize}
        \item[(P1)] $U_{\pm}(t):l^2(\Z^2) \rightarrow l^2(\Z^2)$ and $|U_{\pm}(t)f|_2 \leqslant |f|_2$;
        \item[(P2)] $(U_{+}(t))^*=U_{-}(t)$;
        \item[(P3)] $e^{itL}e^{isL}f = e^{i(t+s)L}f$.
    \end{itemize}
    By (P2), (P3) and Theorem \ref{theorem-main conclusion}, we deduce that
    \begin{equation*}
        \begin{aligned}
            |U_{+}(t)(U_{+}(s))^*g|_{\infty} & =  |U_{+}(t)U_{-}(s)g|_{\infty} \leqslant |e^{itL}e^{-isL}g|_{\infty} = |G(\cdot,t-s)|_{\infty}|g|_1 \\
            %|\mathcal{F}^{-1}(e^{i(t-s)\omega})*g|_{\infty} \leqslant |\mathcal{F}^{-1}(e^{i(t-s)\omega})|_{\infty} |g|_1 \\
            & \lesssim (1+|t-s|)^{-\frac{1}{2}}|g|_1,
            %& \leqslant (1+|t-s|)^{\left(-\frac{3}{2}\right)^-}|g|_1,
        \end{aligned}
    \end{equation*}
    which is the truncated decay estimate. The energy estimate can be obtained by the Plancherel identity, so Lemma \ref{lemma-tao's result} implies 
    \begin{equation}\label{equ-tao1}
    \|e^{itL} f\|_{L^q_t l^r} \lesssim |f|_{2}
    \end{equation} 
    and
    \begin{equation}\label{equ-tao2}
    \left\|\int_0^t e^{i(t-s)L}F(x,s)ds \right\|_{L^q_t l^r} \lesssim \|F\|_{L_t^{\overline{q}'} l^{\overline{r}'}}
    \end{equation}
    for $t,s \in \R$ as long as the index pairs $(q,r)$ and $(\overline{q},\overline{r})$ are $\frac{1}{2}$-admissible, i.e. $q,r \geqslant 2$ and $\frac{1}{q} \leq \frac{1}{2}\left(\frac{1}{2} - \frac{1}{r}\right)$. With (\ref{equ-Duhamel formula}), (\ref{equ-tao1}) and (\ref{equ-tao2}) in hands, we have
    \begin{equation*}
        \begin{aligned}
        \|u\|_{L^q_t l^r} 
        & \leqslant \| e^{itL} f\|_{L^q_t l^r}+\; \left\|\int_0^t e^{i(t-s)L}F(x,s)ds \right\|_{L^q_t l^r} \\
        & \leq C_{q,r,\tilde{q},\tilde{r}} \left(|f|_{2} + \left\|F\right\|_{L_t^{\overline{q}'} l^{\overline{r}'}}\right).
    \end{aligned}
    \end{equation*}
\end{proof}

\begin{remark}\label{remark-strichartz norm}
    We point out that the Strichartz norm can always be bounded by two space-time Lebesgue norms. In fact, for $\gamma=0$, we have
    \begin{equation*}
        ||h||_{\mathcal{S}} = \max\{||h||_{L_t^4 l^{\infty}},||h||_{L_t^{\infty} l^{2}}\}.
    \end{equation*}
    To make it clear, recall that $(q,r)$ is a Strichartz pair if it satisfies $\frac{1}{q} \leq \frac{1}{2}\left(\frac{1}{2}-\frac{1}{r}\right)$. For any such pair $(q,r)$, there exists $2 \leq r_0 \leq r$, such that $\frac{1}{q} = \frac{1}{2}\left(\frac{1}{2}-\frac{1}{r_0}\right)$. Consequently, there is $\theta \in [0,1]$ such that
    \begin{equation*}
        \left(\frac{1}{q},\frac{1}{r_0}\right)=\theta\left(\frac{1}{4},0\right)+(1-\theta)\left(0,\frac{1}{2}\right).
    \end{equation*}
    Then by the embedding $l^{r_0} \subset l^r$ and the log-convexity of $L^p$ norms, we get
    \begin{equation*}
        \begin{aligned}
            ||h||_{L_t^q l^{r}} & \leq ||h||_{L_t^q l^{r_0}} \leq ||h||_{L_t^4 l^{\infty}}^{\theta}||h||_{L_t^{\infty} l^{2}}^{1-\theta} \\
            & \leq \theta ||h||_{L_t^4 l^{\infty}}+ (1-\theta) ||h||_{L_t^{\infty} l^{2}} \\
            & \leq \max\{||h||_{L_t^4 l^{\infty}},||h||_{L_t^{\infty} l^{2}}\}.
        \end{aligned}  
    \end{equation*}
    The third inequality is valid due to the Young's inequality, $ab \leq \frac{1}{p}a^p+\frac{1}{q}b^q$ for $\frac{1}{p}+\frac{1}{q}=1$.
\end{remark}

As an application, we consider the discrete fourth-order Schr\"{o}dinger equation with power type nonlinearity, i.e. $F(u) = \pm |u|^{s-1}u$. In the following theorem we set
\begin{equation}\label{equ-nonlinear index}
    \begin{cases}
        s>5,& \text{if $\gamma=-8$;}\\
        s\geq 5,& \text{if $\gamma=0$ or $-16$;}\\
        s\geq \frac{11}{3},& \text{otherwise}.
    \end{cases}
\end{equation}

\begin{theorem}\label{theorem-nonlinear equation}
    Assume that $d=2$ and $s$ satisfies \eqref{equ-nonlinear index}. If $f$ has sufficiently small norm in $l^2$, i.e. $||f||_{l^2} \leq \varepsilon$ for some $\varepsilon$, then nonlinear equation \eqref{equ-origin fourthorder equation} with $F(u) = \pm |u|^{s-1}u$ has a unique global solution. Furthermore, the solution satisfies
    \begin{equation*}
        ||u||_{\mathcal{S}} \leqslant C \varepsilon.
    \end{equation*}
\end{theorem}

\begin{proof}
    We still prove the conclusion for $\gamma=0$. To apply a standard contraction mapping argument, we consider the metric space
    \begin{equation*}
        \mathcal{M}=\left\{h(x,t):\Z^2 \times \R \rightarrow \C,\, ||h||_{\mathcal{M}}:=||h||_{\mathcal{S}} \leqslant 2C_0||f||_{l^2}\right\}
    \end{equation*}
    with $C_0=\max\{C_{4,\infty,\infty,2},C_{\infty,2,\infty,2}\}$, where $C_{q,r,\overline{q},\overline{r}}$ is the constant in Theorem \ref{theorem-Strichartz estimate}. We define the map $\Lambda$ on $\mathcal{M}$,
    \begin{equation}\label{equ-contraction map}
        \Lambda u(t) = e^{itL}f -i \int_{0}^{t}e^{i(t-\tau)L}F(u(\tau))\,d\tau.
    \end{equation}
    Note that a global solution to \eqref{equ-origin fourthorder equation} is a fixed point of $\Lambda$ on $\mathcal{M}$. By \eqref{equ-contraction map} and Theorem \ref{theorem-Strichartz estimate} (with $\overline{q}'=1,\overline{r}'=2$), we have, for $u \in \mathcal{M}$ and any Strichartz pair $(q,r)$,
    \begin{equation*}
        \begin{aligned}
            ||\Lambda u||_{L^q_t l^r} 
            & \leq \max\{||\Lambda u||_{L_t^4 l^{\infty}},||\Lambda u||_{L_t^{\infty} l^{2}}\} \\
            & \leq C_0 \left(|f|_2+||u||_{L^s_t l^{2s}}^s\right) 
             \leq C_0 |f|_2 + C_0 ||u||_{\mathcal{S}}^s \\
            & \leq C_0 |f|_2 + C_0 (2C_0|f|_{2})^s \leq C_0 |f|_2 + C_0 (2C_0\varepsilon)^s.
        \end{aligned}
    \end{equation*}
    In the third inequality we have used the fact that $(s,2s)$ is a Strichartz pair as long as $s \geq 5$. Thus, for appropriate $\varepsilon$, one has
    \begin{equation*}
        ||\Lambda u||_{\mathcal{S}} \leqslant 2 C_0 |f|_2,
    \end{equation*}
    that is, $\Lambda u \in \mathcal{M}$. Similarly, one can show that $\Lambda$ is a contraction with
    \begin{equation*}
            ||\Lambda u_1 - \Lambda u_2||_{L^q_t l^r}
             \leqslant C_s (||u_1||_{L^s_t l^{2s}}^{s-1}+||u_2||_{L^s_t l^{2s}}^{s-1})||u_1-u_2||_{L^s_t l^{2s}}. 
    \end{equation*}
    The estimate for the solution follows directly.
\end{proof}

\section*{Acknowledgement}

The author would like to express gratitude to Prof. Bobo Hua and Jiqiang Zheng for helpful suggestions, and Cheng Bi for discussions on this problem. 

%B.H. is supported by NSFC, No. 11831004, and by Shanghai Science and Technology Program [Project No. 22JC1400100].

%\bibliography{Bib-sample}
%\bibliographystyle{siam}
%\small

\printbibliography

\end{document}